\documentclass[oneside, a4paper, reqno]{amsart}

\calclayout
\makeatletter 
\def\ifdraft{\ifdim\overfullrule>\z@
        \expandafter\@firstoftwo
    \else
        \expandafter\@secondoftwo
    \fi}
\makeatother 

\def\clap#1{\hbox to 0pt{\hss#1\hss}}

\usepackage[utf8]{inputenc}
\usepackage[T1]{fontenc}
\usepackage{stmaryrd} 
\usepackage{amssymb} 
\usepackage{xcolor}
\usepackage{tikz}
\usepackage{tikz-cd}
\usepackage{mathtools} 
\usepackage{calrsfs}   
\usepackage{extarrows} 

\usepackage[final]{hyperref} 
\hypersetup{
    colorlinks  = true,
    linkcolor   = blue,
    anchorcolor = blue,
    citecolor   = blue,
    filecolor   = blue,
    urlcolor    = blue,
}
\usepackage{cleveref}
\usepackage{graphicx}
\graphicspath{ {./figures/} }
\usepackage{relsize} 
\usepackage{adjustbox}

\usepackage{stackengine}

\usepackage{tabstackengine}
\setstackgap{L}{.75\baselineskip}
\setstackTAB{,}

\ifdraft{
    \usepackage{silence} 
    \usepackage[inline]{showlabels} 
    
}

\DeclareUnicodeCharacter{2211}{\Sigma}
\ifdraft{
    \DeclareUnicodeCharacter{202F}{\,}
}

\usetikzlibrary{shapes.geometric}
\tikzstyle{polyCol} = [fill, cap=round, join=round, line width=30pt]
\definecolor{myred}{rgb}{1,0.2,0.3}
\tikzset{ PO/.style = {dashed,very near start, commutative diagrams/phantom, "\lrcorner", /tikz/commutative diagrams/.cd,dr}} 
\tikzset{ PB/.style = {dashed,very near end, commutative diagrams/phantom, "\ulcorner", /tikz/commutative diagrams/.cd,dr}} 
\tikzset{ labl/.style={anchor=north, rotate=90, inner sep=1mm} }
\tikzset{
  redCirc/.style={
    label={[inner sep=0,minimum size=10mm, circle, fill=myred, opacity=0.2]center:{}}
    }
}

\newtheorem{theorem}{Theorem}[section]
\newtheorem{proposition}[theorem]{Proposition}
\newtheorem{lemma}[theorem]{Lemma}
\newtheorem{corollary}[theorem]{Corollary}
\newtheorem*{theorem*}{Theorem}

\theoremstyle{definition}
\newtheorem*{notation}{Notation}

\newtheorem{definition}[theorem]{Definition}

\newtheorem{setup}[theorem]{Setup}

\definecolor{asparagus}{rgb}{0.53, 0.66, 0.42}

\ifdraft{
    
}{
    
}

\let\mod\relax
\let\Im\relax

\DeclareMathOperator{\Cok}{Cok}
\DeclareMathOperator{\cok}{cok}

\DeclareMathOperator{\End}{End}
\DeclareMathOperator{\Ext}{Ext}
\DeclareMathOperator{\Gen}{Gen}

\DeclareMathOperator{\Hom}{Hom}
\DeclareMathOperator{\id}{id}
\DeclareMathOperator{\Im}{Im}

\DeclareMathOperator{\mod}{mod}

\DeclareMathOperator{\pd}{pd}

\DeclareMathOperator{\Sub}{Sub}

\DeclareMathOperator\D     {\mathsf{D}}      

\DeclareMathAlphabet{\pazocal}{OMS}{zplm}{m}{n}
\newcommand{\cA}{\mathcal{A}}
\newcommand{\cD}{\mathcal{D}}

\newcommand{\cS}{\mathcal{S}}
\newcommand{\cB}{\mathcal{B}}
\newcommand{\cC}{\mathcal{C}}
\newcommand{\cE}{\mathcal{E}}
\newcommand{\cF}{\mathcal{F}}

\newcommand{\cT}{\mathcal{T}}

\newcommand{\pA}{\pazocal{A}}
\newcommand{\pC}{\pazocal{C}}
\newcommand{\pD}{\pazocal{D}}
\newcommand{\pE}{\pazocal{E}}
\newcommand{\pF}{\pazocal{F}}

\newcommand{\pP}{\pazocal{P}}
\newcommand{\pS}{\pazocal{S}}
\newcommand{\pT}{\pazocal{T}}
\newcommand{\pX}{\pazocal{X}}
\newcommand{\pY}{\pazocal{Y}}
\newcommand{\pZ}{\pazocal{Z}}

\newcommand{\NN}{\mathbb{N}}

\newcommand{\ZZ}{\mathbb{Z}}

\newcommand{\xto}{\xrightarrow}

\newcommand{\mono}{\rightarrowtail}
\newcommand{\epi}{\twoheadrightarrow}

\newcommand\xmono[2][]{\ensurestackMath{\mathrel{%
  \stackengine{1pt}{%
    \stackengine{0pt}{\rightarrowtail}{\scriptstyle#2}{O}{c}{F}{F}{S}%
  }{\scriptstyle#1}{U}{c}{F}{F}{S}%
}}}
\newcommand\xepi[2][]{\ensurestackMath{\mathrel{%
  \stackengine{1pt}{%
    \stackengine{0pt}{\twoheadrightarrow}{\scriptstyle#2}{O}{c}{F}{F}{S}%
  }{\scriptstyle#1}{U}{c}{F}{F}{S}%
}}}

\DeclarePairedDelimiterX\filt[1]{\langle}{\rangle}{ #1 }
\DeclarePairedDelimiterX\set[1]{\lbrace}{\rbrace}{ #1 }
\DeclarePairedDelimiterX\Set[1]{\lbrace}{\rbrace}{ #1 }

\newlength\tindent
\renewcommand{\indent}{\hspace*{\tindent}}

\ifdraft{
    \WarningFilter{latex}{Marginpar on page} 
    
}{
}

\newcommand{\dFourTorsScope}[1]{
    \begin{scope}[shift={#1}]
        \dFourTors
}
\newcommand{\dFourTors}[9]{
    \draw [gray, fill=white, fill=#1] (0,3) circle (14pt);
    \draw [gray, fill=white, fill=#2] (1,2) circle (14pt);
    \draw [gray, fill=white, fill=#3] (2,1) circle (14pt);
    \draw [gray, fill=white, fill=#4] (2,3) circle (14pt);
    \draw [gray, fill=white, fill=#5] (3,2) circle (14pt);
    \draw [gray, fill=white, fill=#6] (3,4) circle (14pt);
    \draw [gray, fill=white, fill=#7] (4,3) circle (14pt);
    \draw [gray, fill=white, fill=#8] (5,2) circle (14pt);
    \draw [gray, fill=white, fill=#9] (5,4) circle (14pt);
    \end{scope}
}

\makeatletter
\@namedef{subjclassname@2020}{%
  \textup{2020} Mathematics Subject Classification}
\makeatother

\subjclass[2020]{18E10,18G80}
\keywords{Proper abelian subcategories, Torsion theory, Filtrations}

\title{Filtrations of Torsion Classes in Proper Abelian Subcategories} 
\author{Anders S. Kortegaard}
\address{Department of Mathematics, Aarhus University, Ny Munkegade 118, 8000 Aarhus C, Denmark}
\email{kortegaard@math.au.dk}

\begin{document}

\begin{abstract}
    In an abelian category $\cA$, we can generate torsion pairs from tilting objects of projective dimension $\leq 1$. However, when we look at tilting objects of projective dimension $2$, there is no longer a natural choice of an associated torsion pair.
    Instead of trying to generate a torsion pair, Jensen, Madsen and Su generated a triple of extension closed classes that can filter any objects of $\cA$.
    We generalize this result to proper abelian subcategories.
\end{abstract}

\maketitle

\section{Introduction}
\renewcommand{\thetheorem}{\Alph{theorem}}
\setcounter{theorem}{0}

Let $k$ be a field. 
Given a finite dimensional $k$-algebra $\Lambda$ let $\cA = \mod \Lambda$. 
Given a tilting object $T\in\cA$ of projective dimension $\pd(T)\leq 1$ (see \cite[chap. I.4]{happel1996tilting} for a definition), 
we can construct a torsion pair $(\pT,\pF)$ where $\pT=\Gen(T)$ and $\pF = \pT^\perp$.
For each $x\in\cA$ there exists, by definition, a short exact sequence $t\mono x\epi f$ with $t\in\pT$ and $f\in\pF$.
Another way to describe this is by saying that there exists a filtration $0\subseteq t\subseteq x$ of $x$, where the quotient of the first inclusion $\Cok(0\mono t)\in\cT$ and the quotient of the second inclusion $\Cok(t\mono x)\in\cF$. 
The reason to formulate it in this way will become clear a bit later.

\indent If $T$ is a tilting object of projective dimension $\pd(T) = 2$ there is no longer a natural choice for a torsion pair in $\cA$ associated to $T$, but there will be an associated t-structure in the derived category $\D^b(\cA)$ whose heart is equivalent to $\cB\coloneqq\mod(\End(T)^{op})$ (under the right assumptions, see \cite[sec.~ III.4]{beligiannis2007homological}). Furthermore, $\cB$ will be derived equivalent to $\cA$.
In \cite{jensen2013filtrations} Jensen, Madsen and Su use this derived equivalence to construct three extension closed classes $\pE_0,\pE_1,\pE_2\subseteq\cA$
with $\Hom(\pE_i,\pE_j)=0$ for $i<j$, such that each object in $\cA$ can be filtered with quotients in $\pE_i$.

\begin{theorem*}[{\cite[thm. 2]{jensen2013filtrations}}]
    Given $x\in\cA$ there exists a unique filtration  $0=x_0\subseteq x_1\subseteq x_2\subseteq x_3=x$ such that
    $\Cok(x_{i}\mono x_{i+1})\in \pE_i$, for $i=0,1,2$.
\end{theorem*}

In this article we will give a different construction of the classes $\pE_i$ and generalize this theorem to the setting of proper abelian subcategories.
The concept of proper abelian subcategories is a generalization of hearts of t-structures, introduced by Jørgensen in \cite{jorgensen2022abelian}.
A proper abelian subcategory $\cA$ is an abelian category that sits inside a triangulated category $\cT$ in such a way that short exact sequences in $\cA$ correspond exactly to short triangles in $\cT$ whose objects are in $\cA$, see \Cref{def:pasc}.

\medbreak
Instead of using the derived equivalence to construct the classes $\pE_i$, we will construct them by using proper abelian subcategories $\cA$ and $\cB$ that satisfy the property that $\cB\subseteq\Sigma^2\cA\ast\Sigma\cA\ast\cA$ and
$\cA\subseteq\cB\ast\Sigma^{-1}\cB\ast\Sigma^{-2}\cB$.
With this we can show the following statement.

\begin{theorem}[=\Cref{res:torsion_filtation}]
    \label{intro:main_res}
    Let $\cT$ be a triangulated category, and let $\cA,\cB$ be proper abelian subcategories, where $\cA$ is a noetherian abelian category,
    satisfying that $\cT(\cA, \Sigma^{-i}\cA)=\cT(\cB, \Sigma^{-i}\cB)=0$ for $1\leq i\leq 5$.
    Assume that $\cB\subseteq\Sigma^2\cA\ast\Sigma\cA\ast\cA$ and $\cA\subseteq\cB\ast\Sigma^{-1}\cB\ast\Sigma^{-2}\cB$.
    Then we can define extension closed classes $\pE_0,\pE_1,\pE_2\subseteq\cA$, with $\Hom(\pE_i,\pE_j) = 0$ for $i<j$, such that given $x\in\cA$, there is a filtration of subobjects 
    $0 = x_0 \subseteq x_1 \subseteq x_2 \subseteq x_3 = x$ such that each quotient $x_{i+1}/x_i = \Cok(x_{i}\mono x_{i+1}) \in \pE_i$.
\end{theorem}

Notice that the condition that $\cT(\cA, \Sigma^{-i}\cA)=0$ for $1\leq i \leq 5$ is satisfied if $\cA$ is the heart of a t-structure (see \cite[lem.~3.1]{holm2013sparseness}).
Furthermore, the condition that $\cB\subseteq\Sigma^2\cA\ast\Sigma\cA\ast\cA$
and $\cA\subseteq\cB\ast\Sigma^{-1}\cB\ast\Sigma^{-2}\cB$ 
essentially says that $\cA$ and $\cB$ are not too far apart,
which as an example would be the case if $\cB$ was induced from a tilting object
of projective dimension $\leq 2$ as in \cite{jensen2013filtrations}.

\medbreak
In \Cref{sec:examples} we will apply \Cref{intro:main_res} to an example of proper abelian subcategories that cannot be seen as hearts of t-structures.

\section{Background}
\renewcommand{\thetheorem}{\arabic{section}.\arabic{theorem}}
\setcounter{theorem}{0}

\subsection{Abelian Categories}
\begin{definition}
    Let $\cT$ be a triangulated category. Given full subcategories $\pX,\pZ\subseteq \cT$ define the full subcategory
    \begin{equation*}
        \pX\ast_\cT\pZ = \set{y\in\cT\mid \text{there exists a triangle }x\to y\to z\to \Sigma x \text{ with } x\in\pX,z\in\pZ}.
    \end{equation*}
\end{definition}

Similarly we can define $\ast$ for an abelian category.

\begin{definition}
    Let $\cA$ be an abelian category. Given full subcategories $\pX,\pZ\subseteq \cA$ define the full subcategory
    \begin{equation*}
        \pX\ast_\cA\pZ = \set{y\in\cA\mid \text{there exists a short exact sequence }x\mono y\twoheadrightarrow z \text{ with } x\in\pX,z\in\pZ}.
    \end{equation*}
\end{definition}

\begin{notation}
    We will omit the subscript of $\ast$ if it is clear in which category the operation is performed.
\end{notation}

\noindent It is well-known that the operation $\ast$ is associative,
both in the context of triangulated categories and abelian categories.

\begin{notation}
    Let $\cA$ be an abelian category. Given objects $x,y\in\cA$ and a monomorphism $f:x\mono y$,
    we write $y/x \coloneqq \Cok(f)$.
\end{notation}

\begin{definition}
    Let $\cA$ be an abelian category, and let $S\subseteq \cA$ be a full subcategory. 
    \begin{enumerate}
    \item[$\bullet$] $\Gen_\cA(S) = \set{x \in \cA \mid \text{there exists an epimorphism } s\twoheadrightarrow x \text{ with }s\in S}$.
    \item[$\bullet$] $\Sub_\cA(S) = \set{x \in \cA \mid \text{there exists a monomorphism } x\mono s \text{ with }s\in S}$.
    \item[$\bullet$] $S^{\perp_{\cA}} = \set{x\in \cA \mid \Hom(S, x) = 0}.$
    \item[$\bullet$] ${}^{\perp_\cA}S = \set{x\in \cA \mid \Hom(x, S) = 0}.$
    \item[$\bullet$] $S$ is said to be \emph{extension closed} if $S\ast S\subseteq S$.
    \item[$\bullet$] Given $n\in\NN$ let $(S)_n$ be the following full subcategory
        \begin{equation*}
            (S)_n \coloneqq \left\{a\in\cA \middle| \begin{array}{l} a \text{ has a filtration } 0 = a_0\subseteq a_1\subseteq\cdots\subseteq a_n = a \\\text{s.t. } a_{i+1}/a_{i}\in S\cup \set{0}\end{array}\right\}.
        \end{equation*}
    Then define the \emph{extension closure of $S$} by $\langle S \rangle_\cA \coloneqq \bigcup_{n\in \NN} (S)_n $.
    \end{enumerate}
\end{definition}

\begin{notation}
    The subscripts of $\langle-\rangle, \Gen(-), \Sub(-), (-)^{\perp}$ and ${}^{\perp}(-)$ will be omitted if it is clear in which abelian category the operation is taking place.
\end{notation}

\begin{lemma}
    \label{res:filt_cok_Sn}
    Let $\cA$ be an abelian category, and let $S\subseteq \cA$ be a full subcategory.
    Given $x\in (S)_n$, with corresponding filtration $0=x_0\subseteq x_1 \subseteq \cdots\subseteq x_n = x$, then for all $0\leq i<n$ the inclusion $x_i\subseteq x$ has cokernel $x/x_i \in (S)_{n-i}$.
\end{lemma}
\begin{proof}

    Consider the following diagram of solid arrows, where $\mono$ represents the inclusions given by  the filtration, which we then complete into short exact sequences.
    \begin{equation*}
    \begin{tikzcd}
        x_i \ar[r, tail] \ar[d, equal]&
        x_{n-1} \ar[r, two heads]\ar[d, tail] & 
        x_{n-1}/x_i \ar[d, dashed, tail] \\
        x_i \ar[r, tail]&
        x_{n} \ar[r,two heads]\ar[d, two heads] & 
        x_{n}/x_i \ar[d, dashed, two heads] \\
                           &
        s\ar[r,equal] & 
        s,
    \end{tikzcd}
    \end{equation*}
    
    with $s\in S\cup \{0\}$.
    Using \cite[lem.~3.5]{buhler2010exact} we can fill out this diagram with the dashed arrows,
    such that the third column is a short exact sequence.
    In particular we get an inclusion $x_{n-1}/x_i \subseteq x_n/x_i$, with cokernel $(x_n/x_i)/(x_{n-1}/x_i)\cong s$.
    Using induction we can construct a filtration
    \begin{equation*}
        0 = x_i/x_i \subseteq x_{i+1}/x_i \subseteq \cdots \subseteq x_{n-1}/x_n \subseteq x_n/x_i,
    \end{equation*}

    where $(x_j/x_i)/(x_{j-1}/x_i)\in S$ for $i<j\leq n$.
    In particular, we get that $x_n/x_i \in (S)_{n-i}$.
\end{proof}

\begin{lemma}
    \label{res:ext_of_Sn}
    Let $\cA$ be an abelian category, and let $S\subseteq \cA$ be a full subcategory.
    Given $n,m\in\NN$ then $(S)_{m+n} = (S)_m \ast (S)_n$.
\end{lemma}
\begin{proof}
    It follows from \Cref{res:filt_cok_Sn} that $(S)_{m+n} \subseteq (S)_m \ast (S)_n$,
    thus to show that they are equal, it is enough to show that $(S)_{m+n} \supseteq (S)_m \ast (S)_n$.
    Let $y\in (S)_{m}\ast (S)_n$, this means that there is a short exact sequence
    \begin{equation}
     \label{eq:ses_filt}
    \begin{tikzcd}
        x \ar[r, tail] & y\ar[r, two heads] & z,
    \end{tikzcd}
    \end{equation}

    with $x\in (S)_m$ and $z\in (S)_n$.
    Since $z\in (S)_n$, there is a filtration 
    $0= z_0\subseteq z_1\subseteq \cdots\subseteq z_n = z$, 
    where $z_i/z_{i-1}\in S\cup \set{0}$.
    Using the inclusion $z_{n-1}\subseteq z_n$ together with (\ref{eq:ses_filt}),
    we can construct the following pullback diagram of solid arrows
    
    \begin{equation*}
    \begin{tikzcd}
        x \ar[r, tail]\ar[d, equal] & y_{n-1}\ar[r, two heads]\ar[d, tail]\ar[dr,PB] & z_{n-1}\ar[d, tail]\\
        x \ar[r, tail] & y\ar[r, two heads] \ar[d, dashed, two heads]& z\ar[d, dashed, two heads]\\
        & s\ar[r, equal] & s.
    \end{tikzcd}
    \end{equation*}

    By \cite[prop.~2.12]{buhler2010exact}, the upper right square is bicartesian,
    meaning that the columns can be completed to short exact sequences, as illustrated by the dashed lines,
    such that $y/y_{n-1}\cong z/z_{n-1} \in S\cup\set{0}$.
    Notice that $x\subset y_{n-1}$.
    Using the same trick, an induction argument will construct a filtration $y_0\subseteq y_1\subseteq \cdots\subseteq y_n = y$,
    where $y_{i}/ y_{i-1}\in S\cup \set{0}$. 
    Furthermore, for each $i$ we get a short exact sequence
    \begin{equation*}
    \begin{tikzcd}
        x \ar[r, tail] & y_{i}\ar[r, two heads] & z_{i}.
    \end{tikzcd}
    \end{equation*}

    In particular we get such a short exact sequence for $i=0$, and since $z_i=0$ this gives that $x\cong y_0$.
    Notice that this short exact sequence also gives that $x\subseteq y_i$ for all $i$.
    Combining the filtration we have of $y$ so far, together with that of $x$ we get a filtration
    \begin{equation*}
       0 = x_0\subseteq x_1 \subseteq \cdots \subseteq x_m \subseteq y_1 \subseteq  \cdots \subseteq y_n = y,
    \end{equation*}

    where the cokernel of each inclusion is contained in $S\cup \set{0}$, meaning that $y\in (S)_{m+n}$.
    With this we can conclude that $(S)_{m+n} = (S)_m \ast (S)_n$.
\end{proof}

\begin{corollary}
    Let $\cA$ be an abelian category, and let $S\subseteq \cA$ be a full subcategory.
    Then $\langle S \rangle$ is an extension-closed subcategory.
\end{corollary}
\begin{proof}
    Let $y\in \langle S\rangle \ast \langle S\rangle$.
    This means there is a short exact sequence
    \begin{equation*}
    \begin{tikzcd}
        x\ar[r,tail]&
        y\ar[r,two heads]&
        z
    \end{tikzcd}
    \end{equation*}
    
    with $x,z\in \langle S\rangle = \cup_{n\in \NN} (S)_n$.
    Thus there exists $n,m\in\NN$ such that $x\in (S)_n$ and $z\in (S)_m$.
    \Cref{res:ext_of_Sn} now gives that $y \in (S)_{n+m} \subseteq \langle S\rangle$.
\end{proof}

\begin{lemma}
    \label{res:quotient_subobject_closed_extensios}
    Let $\cA$ be an abelian category, and let $\pX,\pY\subseteq \cA$ be full subcategories.
    \begin{enumerate}
        \item\label{item:qsce:quotient} If $\pX,\pZ$ are closed under quotients, then so is $\pX\ast\pZ$.
        \item\label{item:qsce:subobject} If $\pX,\pZ$ are closed under subobjects, then so is $\pX\ast\pZ$.
    \end{enumerate}
\end{lemma}
\begin{proof}
    (\ref{item:qsce:quotient}) Assume we have an epimorphism $v:y\twoheadrightarrow a$ with $a\in\cA$ and $y\in\pX\ast\pZ$.
    That means there is a diagram with $x\in\pX$ and $z\in\pZ$ and the row short exact.

    \begin{equation*}
    \begin{tikzcd}
        x\rar[rightarrowtail,"f"] &y \rar[twoheadrightarrow,"g"]\dar["v",twoheadrightarrow] & z\\
        & a  &
    \end{tikzcd}
    \end{equation*}
    
    Notice that $vf$ factors over its own image, thus giving a commutative diagram of solid arrows

    \begin{equation*}
    \begin{tikzcd}
        x\rar[rightarrowtail,"f"]\dar[twoheadrightarrow, "u"] &y \rar[twoheadrightarrow,"g"]\dar["v",twoheadrightarrow] & z\dar[twoheadrightarrow, "w", dashed]\\
        \Im(vf)\rar[rightarrowtail,"\alpha"]& a\rar[twoheadrightarrow, "\beta"]  & \Cok(\alpha).
    \end{tikzcd}
    \end{equation*}

    Since $\beta v f = \beta\alpha u = 0$ there exists a morphism $w:z\rightarrow \Cok(\alpha)$, making the diagram above commute.
    Notice that $w$ is an epimorphism since $\beta v$ is an epimorphism. 
    Hence $\Im(vf)\in\pX$ and $\Cok(\alpha)\in \pZ$ and thus
    $a\in\pX\ast\pZ$.

    \medbreak
    (\ref{item:qsce:subobject}) Follows by a similar argument to (\ref{item:qsce:quotient}).
\end{proof}

\begin{lemma}
    \label{res:filt_fac_closed_under_fac}
    Let $\cA$ be an abelian category, and let $S\subseteq \cA$ be a full subcategory.
    Then $\langle\Gen(S)\rangle$ is closed under quotients, and $\langle\Sub(S)\rangle$ is closed under subobjects.
\end{lemma}
\begin{proof}
    This follows directly from \Cref{res:quotient_subobject_closed_extensios} by the use of induction.
\end{proof}

\begin{lemma}
    \label{lem:perp:inclusions}
    Let $\cA$ be an abelian category, and let $S\subseteq \cA$, then the following hold.

    \begin{enumerate}
        \item\label{item:perp:1} $S^\perp = \Gen(S)^\perp$, 
        \item\label{item:perp:2} $S^\perp = \filt{S}^\perp$,
        \item\label{item:perp:3} ${}^\perp S = {}^\perp\Sub(S)$, 
        \item\label{item:perp:4} ${}^\perp S = {}^\perp\filt{S}$.
    \end{enumerate}
\end{lemma}
\begin{proof}
    (\ref{item:perp:1}) The inclusion $S^\perp \supseteq \Gen(S)^\perp$ follows directly from the fact that $S\subseteq \Gen(S)$. For the other inclusion let $x\in S^\perp$, and let $z\in \Gen(S)$. This means that there exists an epimorphism $f: s\twoheadrightarrow z$, with $s\in S$.
    Now assume there is a morphism $h: z\rightarrow x$. Since $x\in S^\perp$ we get that $hf=0$, and since $f$ is an epimorphism, this implies that $h = 0$, and therefore $x\in\Gen(S)^\perp$.

    \begin{equation*}
        \begin{tikzcd}[
            /tikz/column 3/.append style={anchor=base west}
        ]
        s \rar["f", twoheadrightarrow]\ar[dr, "0", dashed]&z\dar["h"] & \hspace{-2em}\in\Gen(S)\\
                                       &x& \hspace{-2em}\in S^{\perp}.
    \end{tikzcd}
    \end{equation*}

    \medbreak
    (\ref{item:perp:2}) The inclusion $S^\perp \supseteq \filt{S}^\perp$ follows directly from the fact that $S\subseteq \filt{S}$. 
    To show the other inclusion, notice that $S^\perp = (S)_1^\perp$, and therefore it will be enough to show that $(S)_{n-1}^\perp \subseteq (S)_{n}^\perp$.
    Let $x\in(S)_{n-1}^\perp$, and let $z\in (S)_{n}$.
    By \Cref{res:ext_of_Sn} there is a short exact sequence $s_1\rightarrowtail z \twoheadrightarrow s_2$, with $s_1\in S$ and $s_2 \in (S)_{n-1}$.
    Assume we have a morphism $f: z\rightarrow x$, that gives us the following diagram
    \begin{equation*}
    \begin{tikzcd}
        s_1\rar[rightarrowtail,"\alpha"]&z \rar[twoheadrightarrow,"\beta"]\dar["f"] & s_2\ar[dl, dashed,"\gamma"]\\
                               & x.
    \end{tikzcd}
    \end{equation*}

    Since $x\in (S)_{n-1}^\perp$ we get that $f\alpha = 0$, and thus there exists a morphism $\gamma:s_2\rightarrow x$ such that $f = \gamma\beta$. 
    However, $s_2\in (S)_{n-1}$ which implies that $\gamma = 0$, hence $f=0$.

    \medbreak
    The proof of (\ref{item:perp:3}) and (\ref{item:perp:4}) is similar to that of (\ref{item:perp:1}) and (\ref{item:perp:2}).
\end{proof}

\begin{definition}
    Let $\cA$ be an abelian category, and let $\pF,\pT\subseteq \cA$ be full subcategories.
    we say that $(\pT, \pF)$ is a \emph{torsion pair} if
    \begin{enumerate}
        \item $\Hom_\cA(\pT,\pF) = 0$,
        \item $\cA = \pT\ast\pF$.
    \end{enumerate}
    Here $\pT$ is called the \emph{torsion part}, and $\pF$ is called the \emph{torsion-free part}.
\end{definition}

We can also define the torsion part and torsion-free part by themselves.

\begin{definition}
    Let $\cA$ be an abelian category, $\pX\subseteq \cA$ a full subcategory, then
    \begin{enumerate}
        \item $\pX$ is called a \emph{torsion class} if it is closed under quotients and extensions,
        \item $\pX$ is called a \emph{torsion-free class} if it is closed under subobjects and extensions.
    \end{enumerate}
\end{definition}

Given a torsion pair $(\pT, \pF)$ in an abelian category $\cA$, it is straightforward to see that the torsion part $\pT$ will be a torsion class, and that the torsion-free part $\pF$ will be a torsion-free class.
However, it is important to note that given a torsion class $\pT'\subseteq \cA$, it does not need to be part of a torsion pair, and similar for torsion-free classes.
If we want every torsion class to be part of a torsion pair, we need some assumptions on $\cA$.

\begin{definition}
    Let $\cA$ be an abelian category, then $\cA$ is said to be \emph{noetherian} if for all objects $x\in \cA$,
    ascending chains of subobjects $x_1\subseteq x_2\subseteq x_3\subseteq \cdots$ of $x$ stabilise.
    That is, there exists $n\in \NN$ such that $x_n = x_{n+i}$ for all $i\in \NN$.
\end{definition}

\begin{theorem}[{\cite[lem. 1.1.3]{polishchuk2007constant}}]
    \label{res:polishchuk:torsion_pairs}
    Let $\cA$ be a noetherian abelian category, and let $\pT\subseteq \cA$ be a torsion class, then $(\pT, \pT^\perp)$ is a torsion pair.
    Similarly, given a torsion-free class $\pF\subseteq\cA$ then $({}^\perp \pF,\pF)$ is a torsion pair.
\end{theorem}

\begin{corollary}
    Let $\cA$ be a noetherian abelian category, and $S\subseteq \cA$.
    Then $(\filt{\Gen(S)}, S^\perp)$ and $({}^\perp S, \filt{\Sub(S)})$ are torsion pairs.
\end{corollary}
\begin{proof}
    This follows directly from \Cref{res:polishchuk:torsion_pairs,res:filt_fac_closed_under_fac,lem:perp:inclusions}.
\end{proof}

\subsection{Proper Abelian Subcategories}

Given an abelian category $\cA$, it sits inside its derived category $\D^b(\cA)$ in such a way that short exact sequences in $\cA$ induce triangles in $\D^b(\cA)$, and short triangles in $\D^b(\cA)$ with objects in $\cA$ come from short exact sequences in $\cA$.
This situation can also be found many other places. Given a t-structure, the corresponding heart sits inside the associated triangulated category with this property.
Simple-minded systems is another way to construct examples with this property that are neither hearts of t-structures nor contained in a derived category. 
To formalize this property we need the following definition.

\begin{definition}[{\cite[def.~1.2]{jorgensen2022abelian}}]
    \label{def:pasc}
    Let $\cT$ be a triangulated category, and let $\cA\subseteq \cT$ be a full additive subcategory. 
    $\cA$ is called a \emph{proper abelian subcategory of $\cT$}, if it is an abelian category in such a way that $x\xmono{\alpha} y\xepi{\beta} z$ is a short exact sequence in $\cA$ if and only if there is a triangle $x\xto{\alpha} y\xto{\beta} z\to \Sigma x$ in $\cT$.
\end{definition}

In a derived category $D = \D^b(\cA)$ the standard heart $\cA$ has no negative self extensions,
i.e. given $a,a'\in\cA$ then $\Hom_D(a,\Sigma^{-n} a')=0$ for $n\geq 1$.
In general, this property is true for an arbitrary heart in a triangulated category (see \cite[lem.~3.1]{holm2013sparseness}).
This is a very useful property of hearts, but it does not hold true for all proper abelian subcategories.
Therefore we need the following definition.

\begin{definition}
    Let $\cA\subseteq\cT$ be a proper abelian subcategory of a triangulated category $\cT$.
    For $n\in\NN$, we say that $\cA$ satisfies $E_n$ if $\Hom_\cA(a,\Sigma^{-i}a')=0$ for all $a,a'\in\cA$ and $1\leq i\leq n$.
\end{definition}

Given two equivalent triangulated categories we can move proper abelian subcategories between them.

\begin{proposition}
    \label{res:pullback-abelian-category}
    Let $\cT,\cD$ be triangulated categories, and $F:\cT\rightarrow \cD$ an equivalence of triangulated categories.
    Given a proper abelian subcategory $\cB\subseteq \cD$, then $F^{-1}(\cB)\subseteq \cT$ is a proper abelian subcategory.
\end{proposition}
\begin{proof}
    This is straightforward to check, thus we omit the proof.
\end{proof}

\section{Filtrations of torsion classes}

The following setup will be assumed throughout this section.
\begin{setup}
    \label{setup:twotilt}
    Let $\cT$ be a triangulated category.
    Let $\cA, \cB\subseteq \cT$ be proper abelian subcategories that satisfy $E_5$,
    such that $\cA\subseteq \cB\ast \Sigma^{-1}\cB \ast \Sigma^{-2}\cB$ and $\cB\subseteq \Sigma^2\cA\ast \Sigma\cA \ast \cA$.
    Furthermore, assume that $\cA$ is noetherian.
    Define
    \begin{enumerate}
        \item[$\bullet$] $\pE_0 = \filt{\Gen_\cA (\cA \cap  \cB)}_\cA$,
        \item[$\bullet$] $\pE_1 = \cA \cap \Sigma^{-1} \cB$,
        \item[$\bullet$] $\pE_2 = \filt{\Sub_\cA (\cA \cap \Sigma^{-2} \cB)}_\cA$.
    \end{enumerate}
\end{setup}

\begin{lemma}
    \label{lem:e0_perp}
    $\pE_0^{\perp_{\cA}} = (\cA \cap  \cB)^{\perp_{\cA}} = \cA\cap (\Sigma^{-1}\cB \ast \Sigma^{-2}\cB)$.
\end{lemma}
\begin{proof}
    \underline{First equality}: This follows directly from \Cref{lem:perp:inclusions}.

    \medbreak
    \underline{Second equality}:
    That $(\cA \cap  \cB)^{\perp_{\cA}} \supseteq \cA\cap (\Sigma^{-1}\cB \ast \Sigma^{-2}\cB)$ follows directly from the fact that $\cB$ satisfies $E_2$.
    For the other inclusion let $a\in (\cA \cap  \cB)^{\perp_{\cA}}$.
    Since $\cA\subseteq \cB\ast \Sigma^{-1}\cB \ast \Sigma^{-2}\cB$, there exists a triangle
    \begin{equation*}
    \begin{tikzcd}
        \Sigma^{-1}x \rar & b\rar["f"] & a \rar& x,
    \end{tikzcd}
    \end{equation*}

    with $b\in\cB$ and $x\in \Sigma^{-1}\cB \ast \Sigma^{-2}\cB$.
    We want to show that $f=0$, which would give that $a$ is a direct summand of $x$,
    and thus $a\in\Sigma^{-1}\cB\ast\Sigma^{-2}\cB$ by \cite[prop. 2.1(1)]{iyama2008mutation}, implying the inclusion we are seeking.
    Since $b\in\cB\subseteq \Sigma^2\cA\ast \Sigma\cA \ast \cA$ there exists a triangle
    \begin{equation*}
    \begin{tikzcd}
        z \rar["g"] & b \rar & \tilde a \rar& \Sigma z,
    \end{tikzcd}
    \end{equation*}
     
    where $\tilde a\in\cA$ and $z \in \Sigma^2\cA\ast\Sigma\cA$. This gives the combined diagram
    \begin{equation*}
    \begin{tikzcd}
        &z\dar["g"]\ar[dl, dashed]\\
        \Sigma^{-1}x \rar & b\rar["f"]\dar & a \rar& x.\\
                          &\tilde a
    \end{tikzcd}
    \end{equation*}
    
    Since $\cA$ satisfies $E_2$ we get that $fg = 0$, and therefore $g$ factors through $\Sigma^{-1} x$, but
    \begin{equation*}
        \Sigma^{-1} x\in \Sigma^{-2} \cB\ast\Sigma^{-3}\cB\subseteq (\cA\ast \Sigma^{-1}\cA\ast\Sigma^{-2}\cA)\ast(\Sigma^{-1}\cA\ast\Sigma^{-2}\cA\ast\Sigma^{-3}\cA).
    \end{equation*}

    Using that $\cA$ satisfies $E_5$ we get that $\Hom(z,\Sigma^{-1} x) = 0$, and therefore $g=0$.
    Thus $b$ is a direct summand of $\tilde a$, in particular $b\in\cA$.
    By assumption $a\in (\cA \cap  \cB)^{\perp_{\cA}}$ meaning that $f=0$, giving the result we want.
\end{proof}

\begin{lemma}
    \label{lem:perp_e2}
    $^{\perp_{\cA}}\pE_2 = {}^{\perp_{\cA}}(\cA\cap\Sigma^{-2}\cB) = \cA\cap(\cB\ast \Sigma^{-1}\cB)$.
\end{lemma}
\begin{proof}
    The proof of this is very similar to that of \Cref{lem:e0_perp} and is therefore omitted.
\end{proof}

\begin{corollary}
    \label{res:torsion_pairs}
    There are torsion pairs
    \begin{enumerate}
        \item[$\bullet$] $(\pE_0,\cA\cap (\Sigma^{-1}\cB\ast \Sigma^{-2}\cB))$,
        \item[$\bullet$] $(\cA\cap (\cB\ast \Sigma^{-1}\cB) , \pE_2)$.
    \end{enumerate}
\end{corollary}
\begin{proof}
    This follows directly from \Cref{lem:e0_perp,lem:perp_e2} and \Cref{res:polishchuk:torsion_pairs}.
\end{proof}

\begin{lemma}
    There is a filtration of torsion classes $0\subseteq \pE_0 \subseteq \cA\cap(\cB\ast\Sigma^{-1}\cB)\subseteq \cA$.
\end{lemma}
\begin{proof}
        The only inclusion that it is necessary to check is the second inclusion.
        Notice that $\cA \cap  \cB \subseteq \cA\cap (\cB\ast \Sigma^{-1}\cB)$.
        By \Cref{res:torsion_pairs} we get that $\cA\cap (\cB\ast \Sigma^{-1}\cB)$ is a torsion class, and is therefore closed under taking quotients and extensions.
        Thus it follows directly that $\pE_0 = \filt{\Gen_\cA (\cA \cap  \cB)}_\cA \subseteq \cA\cap(\cB\ast\Sigma^{-1}\cB)$.
\end{proof}

\begin{corollary}
    \label{res:torsion_filtation}
    Let $x\in\cA$, then up to isomorphism there exists a unique filtration of subobjects 
    $0 = x_0 \subseteq x_1 \subseteq x_2 \subseteq x_3 = x$ such that each quotient $x_{i+1}/x_i = \cok(x_i\mono x_{i+1})\in \pE_i$.
\end{corollary}
\begin{proof}
     By \Cref{res:torsion_pairs} we get that $(\cA\cap (\cB\ast \Sigma^{-1}\cB) , \pE_2)$ is a torsion pair in $\cA$.
     Thus there exists a short exact sequence
     \begin{equation*}
     \begin{tikzcd}
         x_2 \rar[rightarrowtail] & x \rar[twoheadrightarrow]& e_2
     \end{tikzcd}
     \end{equation*}

     with $x_2\in\cA\cap (\cB\ast \Sigma^{-1}\cB)$ and $e_2\in\pE_2$.
     \Cref{res:torsion_pairs} also says that there is a torsion pair $(\pE_0,A\cap (\Sigma^{-1}\cB\ast \Sigma^{-2}\cB))$.
     Therefore, there exists a short exact sequence

     \begin{equation*}
     \begin{tikzcd}
         x_1 \rar[rightarrowtail] & x_2 \rar[twoheadrightarrow]& e_1
     \end{tikzcd}
     \end{equation*}

     with $x_1\in\pE_0$ and $e_1\in\cA\cap(\Sigma^{-1}\cB\ast \Sigma^{-2}\cB)$.
     Notice that since $\cA\cap (\cB\ast \Sigma^{-1}\cB)$ is a torsion class we get that 
     $e_1\in \cA\cap (\cB\ast \Sigma^{-1}\cB)\cap(\Sigma^{-1}\cB\ast \Sigma^{-2}\cB)=\cA\cap \Sigma^{-1}\cB = \pE_1$ by \cite[lem.~2.2(ii)]{jorgensen2021proper}.
     
     \indent The uniqueness follows directly from the uniqueness of the short exact sequences corresponding to torsion pairs.
\end{proof}

\begin{definition}
    A tuple $(\pS_0, \pS_1, \pS_2)$ of full subcategories in an abelian category $\pA$
    is called a \emph{torsion triple} if
    \begin{enumerate}
        \item $\Hom_{\pA}(\pS_i,\pS_j) = 0$ for $i<j$,
        \item $\pA = \pS_0 \ast \pS_1 \ast \pS_2$.
    \end{enumerate}
\end{definition}

\begin{corollary}
    $(\pE_0,\pE_1,\pE_2)$ is a torsion triple.
\end{corollary}
\begin{proof}
    That $\Hom(\pE_i, \pE_j) = 0$ for $i<j$ follows directly from the fact that $\cB$ satisfies $E_2$.
    For the second condition, let $x\in \pA$, then by \Cref{res:torsion_filtation} there is a filtration
    $0 = x_0 \subseteq x_1 \subseteq x_2 \subseteq x_3 = x$ such that each quotient $x_{i+1}/x_i\in \pE_i$.
    Thus there exist short exact sequences
    \begin{equation*}
    \begin{tikzcd}
        x_2 \ar[r,tail] &
        x \ar[r,tail] &
        e_2
    \end{tikzcd}
    \qquad\text{and}\qquad
    \begin{tikzcd}
        e_0 \ar[r,tail] &
        x_2 \ar[r,tail] &
        e_1,
    \end{tikzcd}
    \end{equation*}

    where $e_i\in \pE_i$. 
    This means that $x_2\in \pE_0\ast \pE_1$ and thus $x \in \pE_0\ast \pE_1\ast\pE_2$.
\end{proof}

\begin{proposition}
    There is a bijection
    \begin{align*}
        \set{\text{torsion triples in $\cA$}}
        &\xlongrightarrow{\Phi}
        \left\{\begin{array}{l}\text{Pairs of torsion pairs } [(\pT,\pF), (\tilde\pT,\tilde\pF)]\\ \text{ in } \cA \text{ satisfying } \pT\subseteq\tilde\pT\end{array}\right\}\\
        (\pS_0, \pS_1, \pS_2)&\longmapsto (\pS_0, \pS_1 \ast \pS_2), (\pS_0 \ast \pS_1, \pS_2)\\
        (\pT, \pF\cap \tilde\pT,\tilde \pF)&\longmapsfrom [(\pT,\pF), (\tilde\pT,\tilde\pF)].
    \end{align*}
\end{proposition}
\begin{proof}
    Denote the potential inverse for $\Phi$ by $\Phi'$.
    It is straightforward to check that the maps take values in the relevant sets.

    \medbreak
    \indent Let us check that $\Phi\Phi' = \id$.
    Let $[(\pT,\pF), (\pT',\pF')]$ be a pair of torsion pairs such that $\pT\subseteq \pT'$.
    Then
    \begin{equation*}
        \Phi\Phi'((\pT,\pF), (\pT',\pF')) 
        = \Phi(\pT, \pF\cap \pT',\pF') 
        = [ (\pT, (\pF\cap \pT')\ast\pF'), (\pT\ast (\pF\cap \pT'),\pF')].
    \end{equation*}

    Thus we need to check that $(\pF\cap \pT')\ast\pF' = \pF$ and  $\pT\ast (\pF\cap \pT') = \pT'$.
    We check the first one of these, and the other one can be shown by a similar argument.
    Since $\pF' \subseteq \pF$ we get that $(\pF\cap \pT')\ast\pF' \subseteq \pF$.
    To see the other inclusion, let $x\in \pF$, then since $(\pT', \pF')$ is a torsion pair, there is a short exact sequence

    \begin{equation*}
        t'\mono x\epi f',
    \end{equation*}

    with $f'\in\pF'$ and $t'\in\pT'$.
    However, since torsion-free classes are closed under subobjects, we get that $t'\in\pF$.
    Thus $t'\in\pF\cap\pT'$, and therefore $x'\in (\pF\cap \pT')\ast\pF'$.

    \medbreak
    \indent Let us check that $\Phi'\Phi = \id$. 
    Let $(\pS_0, \pS_1, \pS_2)$ be a torsion triple.
    Then
    \begin{equation*}
        \Phi'\Phi(\pS_0, \pS_1, \pS_2) 
        = \Phi'((\pS_0, \pS_1 \ast \pS_2), (\pS_0 \ast \pS_1, \pS_2))
        = (\pS_0, (\pS_1 \ast \pS_2) \cap (\pS_0 \ast \pS_1), \pS_2).
    \end{equation*}

    We therefore need to check if $(\pS_1 \ast \pS_2) \cap (\pS_0 \ast \pS_1) = \pS_1$.
    It is straightforward to see that the inclusion $\supseteq$ is satisfied.
    For the other inclusion, let $x\in(\pS_1 \ast \pS_2) \cap (\pS_0 \ast \pS_1)$.
    Thus there are two short exact sequences
    $s_1\mono x\epi s_2$ and $s_0\mono x\epi s_1'$, with $s_i \in \pS_i$ and $s_1'\in \pS_1$.
    Consider the following diagram of solid arrows.
    \begin{equation*}
    \begin{tikzcd}
        s_0 \ar[r, "f"] \ar[d, dashed, "\eta"] & 
        x \ar[r,"g"]\ar[d, equal] & 
        s_1'\\
        s_1 \ar[r,"\alpha"] & 
        x \ar[r, "\beta"] & 
        s_2.
    \end{tikzcd}
    \end{equation*}
    
    Notice that $\beta f = 0$ due to the definition of a torsion triple.
    Thus there exists a morphism $\eta:s_0\to s_1$ such that $f = \alpha \eta$.
    However, $\eta = 0$ due to the same definition, implying that $f = 0$.
    Thus, $g$ is an isomorphism, making $x\cong s_1'\in S_1$.
\end{proof}

\section{Examples}
\label{sec:examples}

\subsection{Jensen-Madsen-Su}
Let $k$ be a field.
Let $\cA$ be a noetherian abelian category of the type studied in \cite{jensen2013filtrations}, that is, $\cA$ is either the module category of a finite-dimensional $k$-algebra,
or it is a noetherian abelian $k$-category with finite homological dimension and Hom-finite derived category $\D^b(\cA)$, such that there is a locally noetherian abelian Grothendieck $k$-category $\cA'$ with finite homological dimension such that $\cA\subseteq \cA'$ is the subcategory of noetherian objects (see \cite[sec.~0]{jensen2013filtrations}).

\indent Now consider a tilting object $T\in\cA$ of homological dimension 2.
By this, we mean an object $T\in\cA$ that induces a derived equivalence 
\begin{equation*}
    F=R\Hom(T,-):\D^b(\cA)\rightarrow \D^b(\widehat \cB),
\end{equation*}

\noindent where $\widehat\cB = \mod(\End(T)^{op})$, such that $H^iFX = \Ext^i(T,X) = 0$ for all $X\in\cA$ and $i\geq 3$.
Denote the quasi-inverse functor by $G: \D^b(\widehat\cB)\rightarrow \D^b(\cA)$.
Notice that $\widehat\cB$ sits inside $\D^b(\widehat\cB)$ as a proper abelian subcategory,
and since $F$ is a derived equivalence we can pull $\widehat\cB$ back to be considered as a proper abelian subcategory of $\D^b(\cA)$, see \Cref{res:pullback-abelian-category}.
Let $\cB = F^{-1}(\widehat\cB)$.

\indent To show that we are in a setup similar to that of \Cref{setup:twotilt}, we need the following lemma.

\begin{lemma}
    Using the notation from above, we have that $\cA\subseteq \cB\ast\Sigma^{-1}\cB\ast\Sigma^{-2}\cB$ and $\cB\subseteq\Sigma^2\cA\ast\Sigma\cA\ast\cA$.
\end{lemma}
\begin{proof}
    We start by proving the first inclusion.
    Let $x\in\cA$.
    Using the assumption that $T$ has projective dimension $2$, it follows that $H^i R\Hom(T,x) = \Ext^i(T,x) = 0$ for $i\neq 0,1,2$.
    Thus by the use of soft truncations one can see that $R\Hom(T,x)$ is equivalent to a three term complex concentrated in cohomological degrees $0,1,2$,
    i.e. $x\in\cB\ast\Sigma^{-1}\cB\ast\Sigma^{-2}\cB$.

    \medbreak
    \noindent To show the second inclusion, recall that $\cA$ is the heart of the standard t-structure $(\pD_{\geq 0}, \pD_{< 0})$ in $\D^b(\cA)$.
    For $i\in\ZZ$ denote $\pD_{\geq i} = \Sigma^i \pD_{\geq 0}$, and $\pD_{< i} = \Sigma^i \pD_{< 0}$.
    Similarly, $\widehat \cB$ is the heart of the standard t-structure $(\pP_{\geq 0}, \pP_{< 0})$ in $\D^b(\widehat \cB)$.
    For $i\in\ZZ$ denote $\pP_{\geq i} = \Sigma^i \pP_{\geq 0}$, and $\pP_{< i} = \Sigma^i \pP_{< 0}$.
    Since $T$ has projective dimension $2$ is follows that $F(\pD_{\geq 0}) \subseteq \pP_{\geq -2}$.
    Thus 
    \begin{equation}
        \label{eq:impl1}
        \begin{aligned}
        F(\pD_{\geq 0})^\perp \supseteq \pP_{\geq -2}^\perp
        \quad&\Longrightarrow\quad F(\pD_{\geq 0}^\perp) \supseteq \pP_{\geq -2}^\perp
             &&\Longrightarrow\quad F(\pD_{< 0}) \supseteq \pP_{< -2}
            \\
             &\Longrightarrow\quad \pD_{<0} \supseteq G(\pP_{<-2})
             &&\Longrightarrow\quad \pD_{<3} \supseteq G(\pP_{<1}),
        \end{aligned}
    \end{equation}

    \noindent where the second implication follows from the fact that t-structures are torsion pairs,
    and the last implication follows by applying $\Sigma^3$.
    A similar calculation can be done for the torsion-free parts.
    \begin{equation}
        \label{eq:impl2}
        \begin{aligned}
            F(\pD_{< 0}) \subseteq \pP_{< 0}
            \quad&\Longrightarrow\quad  {}^\perp F(\pD_{< 0}) \supseteq {}^\perp\pP_{<0}
            &&\Longrightarrow\quad F({}^\perp \pD_{< 0}) \supseteq {}^\perp\pP_{<0}\\
            &\Longrightarrow\quad  F(\pD_{\geq 0}) \supseteq \pP_{\geq 0}.
            &&\Longrightarrow\quad \pD_{\geq 0} \supseteq G(\pP_{\geq 0}).
        \end{aligned}
    \end{equation}

    Combining (\ref{eq:impl1}) and (\ref{eq:impl2}) now gives that
    \begin{equation*}
        G(\widehat \cB) 
        \subseteq G(\pP_{\geq 0})\cap G(\pP_{<1})
        \subseteq \pD_{\geq 0} \cap \pD_{<3}.
    \end{equation*}

    This means that objects in $\cB = G(\widehat \cB)$ have homology concentrated in homological degrees $0,1,2$.
    Thus, $\cB \subseteq \Sigma^2\cA \ast \Sigma \cA \ast \cA$.
\end{proof}

\medbreak

In \cite{jensen2013filtrations} Jensen, Madsen and Su define collections of objects 
\[
     \pF^i = \set{x\in\cA\mid H^jF(x) = 0 \text{ for } j\neq i},
\]

for $i\geq 0$. Note that $\pF^i = 0$ for $i\geq 3$.
It is straightforward to check that $\pF^i=\cA\cap\Sigma^{-i}\cB$.
From this they build three other collections of objects $\cE^i$, which by \cite[lem. 17 \& 22]{jensen2013filtrations} can be described as $\pE^0 = \filt{\Gen(\pF^0)}$, $\pE^1 = \pF^1$ and $\pE^2 = \filt{\Sub(\pF^2)}$.
This places us in the situation of \Cref{setup:twotilt} so \cite[thm.~2]{jensen2013filtrations} follows from
our \Cref{res:torsion_filtation}.

\subsection{Negative Cluster Categories}

Negative cluster categories are examples of triangulated categories in which there are proper abelian subcategories, none of which are hearts of t-structures.
Furthermore, these categories also have the advantage that there is a full combinatorial model describing some of them, see \cite[sec. 10]{coelhosimoes2016torsion}.
Here we will give a short introduction.

\medbreak 
Let $n,w\in\NN=\{1,2,3,\cdots\}$, and define the \emph{negative cluster category} as the orbit category
\begin{equation*}
        \pC_{-w}(A_n) \coloneqq \D^b(kA_n)/\Sigma^{w+1}\tau,
\end{equation*}

where $\tau$ refers to the Auslander-Reiten translation. 
This triangulated category is $-w$-Calabi--Yau (see \cite[sec.~4, sec.~8.4]{keller2005triangulated}), which means that $\Sigma^{-w}$ is a Serre functor.
Let $N=(w+1)(n+1)-2$ and consider the $N$-gon, say $\pP_N$.
Labelling the corners by $0,...,N-1$ anticlockwise we can denote each diagonal in $\pP_N$ by a pair of numbers $(a,b)$ with $a<b$. 
We say that the diagonal $(a,b)$ is \emph{admissible} if $w+1\mid b-a+1$.
There is a one to one correspondence between indecomposable objects in $\pC_{-w}(A_n)$ and admissible diagonals in the $N$-gon.

\medbreak
One way to find proper abelian subcategories in this setting is to create them from simple minded systems (see \cite[def.~1.2]{jorgensen2022abelian} for a definition). 
Given a simple minded system $\cS$, a proper abelian subcategory is generated by taking the extension closure $\filt{\cS}$, see \cite[thm.~A]{jorgensen2022abelian}.
In the setting of negative cluster categories, all simple minded systems can be classified combinatorially, see \cite[prop.~2.13]{coelhosimoes2020simple} and \cite[thm.~6.5]{coelhosimoes2015hom}.
In $\pC_{-w}(A_n)$ a collection of $n$ admissible diagonals corresponds to a $w$-simple minded system if no two diagonals in the collection cross and there are no two diagonals sharing an endpoint.
See \Cref{fig:example:sms} for two examples of such collections.

\indent Furthermore, it is possible to describe triangles and morphisms between objects using the combinatorial model.
All of this can be found in \cite[sec. 10]{coelhosimoes2016torsion}.

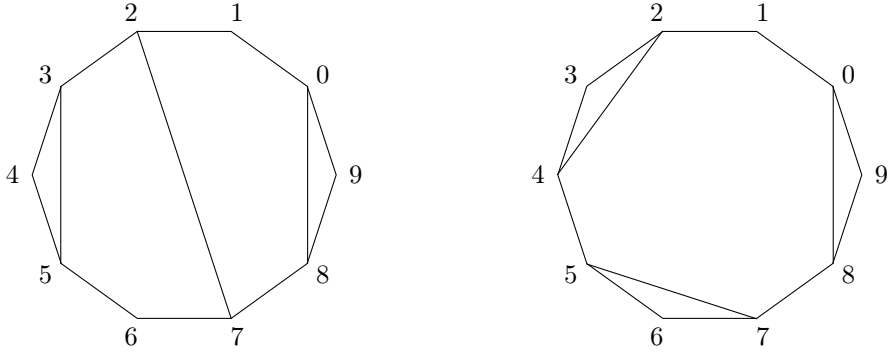
\begin{figure}[ht]
    \begin{tikzpicture}
        \node[regular polygon,regular polygon sides=10,minimum size=4cm,draw] (a){};
        \node[regular polygon,regular polygon sides=10,minimum size=4.5cm,draw, opacity=0] (b){};
        \foreach \x in {1,...,9}{\node[] at (b.corner \x) {\x};}
        \node[] at (b.corner 10) {0};
        \draw[] (a.corner 3) -- (a.corner 5);
        \draw[] (a.corner 7) -- (a.corner 2);
        \draw[] (a.corner 8) -- (a.corner 10);
    \end{tikzpicture}
    \hspace{5em}
    \begin{tikzpicture}
        \node[regular polygon,regular polygon sides=10,minimum size=4cm,draw] (a){};
        \node[regular polygon,regular polygon sides=10,minimum size=4.5cm,draw, opacity=0] (b){};
        \foreach \x in {1,...,9}{\node[] at (b.corner \x) {\x};}
        \node[] at (b.corner 10) {0};
        \draw[] (a.corner 4) -- (a.corner 2);
        \draw[] (a.corner 7) -- (a.corner 5);
        \draw[] (a.corner 8) -- (a.corner 10);
    \end{tikzpicture}
    \caption{two simple-minded systems in $\pC_{-2}(A_3)$.}
    \label{fig:example:sms}
\end{figure}

\medbreak
\textbf{The example.}
Let $w = 6$, $n = 5$, and consider the negative cluster category $\pC_{-w}(A_n)$ which we can work with combinatorially as described above, using an $N$-gon, where $N = (w+1)(n+1)-2 = 40$.
In \Cref{fig:ncc65_tt} we can see a segment of the AR quiver of $\pC_{-w}(A_n)$.
Now consider the following two simple minded systems
\begin{align*}
    \cS_\cA &= \set{( 28, 34 ), ( 14, 20 ), ( 21, 27 ), ( 1, 7 ), ( 0, 13 )} \\
    \cS_\cB &= \set{( 23, 29 ), ( 7, 13 ), ( 22, 35 ), ( 1, 14 ), ( 15, 21 )}.
\end{align*}

Using these we construct proper abelian subcategories $\cA = \filt{\cS_\cA}$ and $\cB=\filt{\cS_\cB}$, see \Cref{fig:ncc65_tt}.
It is straightforward to check that both $\cA$ and $\cB$ satisfy $E_5$.
Similarly, it is straightforward to check that $\cA\subseteq \cB\ast\Sigma^{-1}\cB\ast\Sigma^{-2}\cB$ and $\cB\subseteq \Sigma^{2}\cA\ast\Sigma^{1}\cA\ast\cA$.
Lastly for \Cref{setup:twotilt} to be satisfied, $\cA$ needs to be noetherian.
However, notice that $\cA$ is isomorphic to the module category of a path algebra, that is $\cA\cong\mod kQ$ where $Q$ is the quiver
\begin{equation*}
Q:\quad \begin{tikzcd}
    1\rar&2\rar&3\rar&4&5\lar.
\end{tikzcd}
\end{equation*}

Thus we get that $\cA$ is a noetherian abelian category, and thereby \Cref{setup:twotilt} is satisfied.
With this, the intersections needed can be described as
\begin{align*}
    \pE_0 &= \filt{\Gen_\cA (\cA \cap  \cB)}_\cA &&= \set{(1,7),(7,13)},\\
    \pE_1 &= \cA \cap \Sigma^{-1} \cB &&= \set{ ( 0, 34 ), ( 0, 20 ), ( 14, 20 ), ( 14, 34 ), ( 21, 34 ), ( 0, 13 ), ( 28, 34 ) },\\
    \pE_2 &= \filt{\Sub_\cA (\cA \cap \Sigma^{-2} \cB)}_\cA &&= \set{(21,27)},
\end{align*}
see \Cref{fig:ncc65_tt}.
This means that we are in a setup where \Cref{res:torsion_filtation} can be used. 
To see a specific example of a filtration of an element consider $x=(7,27)\in\cA$.
This element has the filtration 
\begin{equation*}
0\subseteq (7,13)\subseteq (7,20)\subseteq (7,27) = x.
\end{equation*}

Using the combinatorial model for $\pC_{-w}(A_n)$ we can see that there are short triangles,

\begin{equation*}
\begin{tikzcd}
    (7,13)\rar[]&(7,20)\rar[]&(14,20)
\end{tikzcd}
\quad\text{and}\quad
\begin{tikzcd}
    (7,20)\rar[]&(7,27)\rar[]&(21,27).
\end{tikzcd}
\end{equation*}

Since all the elements of the two short triangles are in $\cA$, this implies that we have corresponding short exact sequences
\begin{equation*}
\begin{tikzcd}
    (7,13)\rar[rightarrowtail]&(7,20)\rar[twoheadrightarrow]&(14,20)
\end{tikzcd}
\quad\text{and}\quad
\begin{tikzcd}
    (7,20)\rar[rightarrowtail]&(7,27)\rar[twoheadrightarrow]&(21,27),
\end{tikzcd}
\end{equation*}

where $(7,13)\in\pE_0$, $(14,20)\in\pE_1$, and $(21,27)\in\pE_2$.
Thus the given filtration of $x$ is indeed of the form stated by \Cref{res:torsion_filtation}.

\begin{figure}
\begin{center}
    \includegraphics[scale=.45]{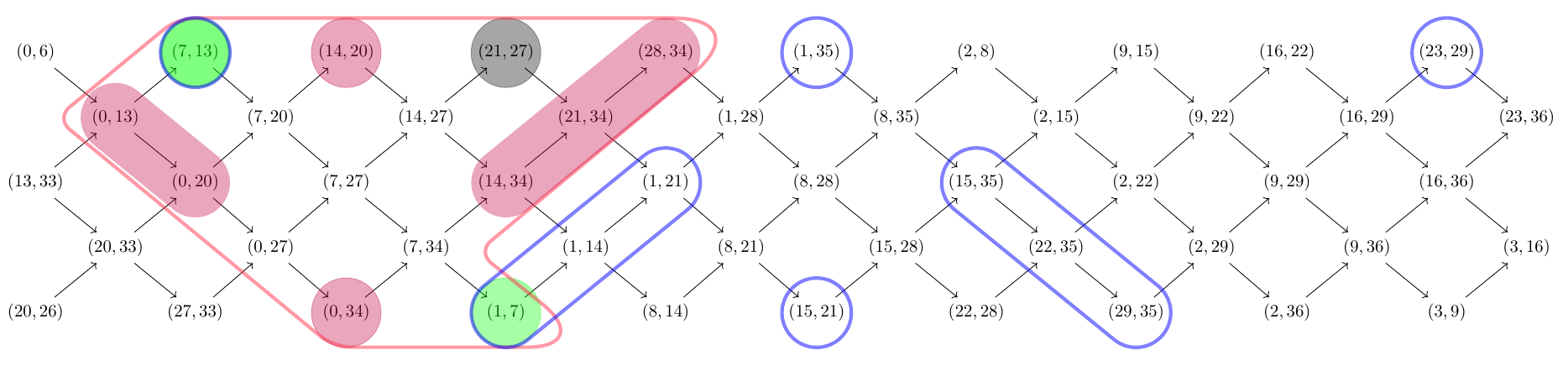}
    \caption{This is a segment of the AR quiver of $\cC_{-6}(A_5)$. The objects in $\cA$ are the ones surrounded by a red line, and the objects of $\cB$ are the ones surrounded by a blue line.
    $\pE_0$ is indicated by a green fill, $\pE_1$ is indicated by a red fill, and $\pE_2$ is indicated by a gray fill.}
    \label{fig:ncc65_tt}
\end{center}
\end{figure}

\leavevmode

\medbreak

\textit{Acknowledgement.}
I would like to thank my supervisor Peter Jørgensen for all of his helpful guidance.

\indent This project was supported by grant no. DNRF156 from the Danish National Research Foundation.

\end{document}